\documentclass[preprint,12pt]{elsarticle}

\usepackage{graphicx}
\usepackage{amssymb}

\usepackage[all]{xy}
\usepackage{amsmath,amsfonts,amssymb,amscd,amsthm,xspace}
\usepackage[utf8]{inputenc}
\usepackage{relsize}
\newtheorem{theorem}{Theorem}
\newtheorem{definition}{Definition}

\newtheorem{remark}{Remark}
\newtheorem{corollary}{Corollary}

\journal{Journal Name}

\begin{document}

\begin{frontmatter}


\title{A Shear Stress Reynolds' Limit Formula}

 \author{C. V. Valencia-Negrete \fnref{label2}}
 \ead{carla.valencia@ibero.mx}
 \fntext[label2]{Universidad Iberoamericana A. C.}
 \address{Prolongaci\'{o}n Paseo de la Reforma 880,
              Mexico City ---01219, MEXICO \fnref{label2}}




\begin{abstract}
Historically, meteorological and climate studies have been prompted by the need for understanding
precipitation to have better logistics in food production. Despite all efforts, 
nonlinearity in atmosphere dynamics is still a source of uncertainty.
On the other hand, aeronautical science studies the boundary layer separation
through the \emph{shear stress}.
In this work, a mathematical interpretation of methods in classical aerodynamics theory in terms of
successive layers of \emph{diffeomorphisms} over \emph{Lipschitz domains}  
allows us to estimate the boundary layer's
\emph{shear stress}, $\tau^{*}_d$ and $\tau^{*}_m$, in dry and humid atmospheric conditions
without assuming that there is not a
convective derivative term in the conservation of momentum equation 
or that the gaseous boundary layer is incompressible:
\[
\tau^{*}_d = \frac{U}{h}\ \left(1-\frac{U^2}{2c_{pd}\ T_0}\right)^{19/25},
\hspace{7pt}
\tau^{*}_m = \frac{U}{h}\ \left(1-\frac{U^2}{2c_{pm}\ T_0}\right)^{19/25},\]
where $h$ is the boundary layer's height, 
$T_0$ is the surface temperature,
$U$ is the \emph{free stream velocity};
$c_{pd}$ is the \emph{specific heat at constant pressure for dry air}
and $c_{pm}$ is the \emph{specific heat at constant pressure for moist air}.
Furthermore, 
if $\hat{R}_m$ is a \emph{gas constant for moist air}
and $p_0$ is the pressure at the surface,
the density
$\rho
 \hspace{2pt}  \cong \hspace{2pt}
      p_0 \hspace{2pt} T_0^{\frac{2b}{b-1}-1}
     \hspace{2pt} \hat{R}_{m}^{-1} \hspace{2pt}
      \left[1-\left(U^2/2c_{ph}T_0\right)\right]^{\frac{b}{(b-1)}-1}$
      for $b=1.405$.
Moreover, this opens the possibility of finding a different deterministic family of atmosphere natural convection models.
\end{abstract}


\end{frontmatter}


\section{Introduction}
\label{S:1}

This work suggests that a shear stress value obtained from limit formulas where
the deterministic model is approximated but not oversimplified would give us a way to 
better evaluate the conditions under which there is an atmospheric boundary layer separation, and thus, 
convective clouds formation. 
Previous work by the author showed that a Reynolds' Limit Formula could be deduced
from Dorodnitzyn's gaseous boundary layer model to overcome
the inherited Navier-Stokes non-linearity in its convective derivative term \cite{val08}
with the application of Bayada and Chambat's diffeomorphism  \cite{Bayada1986}.
For more details on the statement of the small parameter problem in terms of an incompressible field, see \cite{val08}.

Dorodnitzyn reduced the original system of seven equations for seven variables to a quasi-linear problem
for a transformation of the shear stress in a new domain. 
This work offers a mathematical proof of Dorodnitzyn's deduction in Theorem~\ref{theo:22}.
Surely, there exists a mathematical formalization preceding from the one given here, but the author could not find it in the literature.
This might be a consequence of the fact that Dorodnitzyn's work of the subsequent years is \emph{partially classified} \cite[p. 1973]{kerimov10}.

From the small parameter point of view, in $1886$, Reynolds published:
``\emph{On the Theory of Lubrication and Its Application to Mr. Beauchamp Tower's
Experiments, Including an Experimental Determination of the Viscosity of Olive Oil}"
\cite{Reynolds1886}, where he gave the formula to study fluid behavior when it moves in a narrow space between two plates.
Reynolds' Limit Formula was effectively used for a hundred years before
Guy Bayada and Mich\`{e}le Chambat \cite{Bayada1986}
formally proved that this was indeed a limit formula for the Stokes' Equations 
when the small parameter of a proportion of the boundary layer's
height and its length tend to zero, at $1986$. 
In $2009$, Chupin and Sart demonstrated that
that the compressible Reynolds equation is an approximation of compressible
Navier-Stokes equations  \cite{ChupinSart}. 
They mention that there seems to be only one noticeable work of this type of problem 
for a thin domain filled with gas \cite{ChupinSart} found
in Eduard Marusic-Paloka and Maja Starcevic's results
\cite{MARUSICPALOKA20104565} \cite{MARUSICPALOKA2005534}. 

The proportion $L>\!>\!>h>0$ allows the introduction of a small parameter $\epsilon = h/L$
and the application of Bayada's change of variables \cite{Bayada1986} to obtain a Reynolds' Limit Formula.
Theorem~\ref{theo:33}
gives a demonstration of a Reynolds' Limit Formula for Dorodnitzyn's shear stress quasi-linear problem.
The corresponding Reynolds' Limit Formula for the shear stress deduction was accepted
for an Oral Presentation at the $10$th European Nonlinear Dynamics Conference (ENOC 2020), 
accompanied by an extended abstract that will be published in its proceedings.
The complete proofs and shear stress approximations for dry and humid atmospheric conditions 
are presented here, in Theorem~\ref{theo:33}, as a new additional result.
In particular, it
justifies
the wide use of the free stream velocity as a good approximation
of horizontal velocity near the Earth's surface in meteorology.

\section{Dorodnitzyn's shear stress statement of the problem}

The quasi-linear statement of the original problem in terms of the shear stress is obtained by a series of two essential steps.
First, Theorem~\ref{theo:11} shows that the original problem
has a simplified expression as a system of just one condition 
for the \emph{stream function} $\psi$ taken over the polygon $\Pi=\mathbf{s}(R)$
in terms of  Dorodnitzyn's change of coordinates $\mathbf{s}(x,y)=(\ell,s)$
of the original rectangular domain $R$,
where the convective derivative has an incompressible form.
Second, Theorem~\ref{theo:22} gives a formal proof of how this system can be written in terms of
a transformation that takes the original shear stress to a new domain, an infinite strip band 
$$S \colon=
\left\{(\ell,z)\in \mathbb{R}^2 \hspace{4pt}|\hspace{4pt} (\ell,s)\in \Pi, 
\hspace{5pt}and \hspace{5pt} z =s/\ell^{1/2}\in (0,\infty)\right\},$$
\noindent
following a composition of the original \emph{stream fuction} with Dorodnitzyn's 
diffeomorphism $\mathbf{s}$, and with Blasius' adapted height normalization 
$z=s/\ell^{1/2}$ \cite{Blasius1908} :
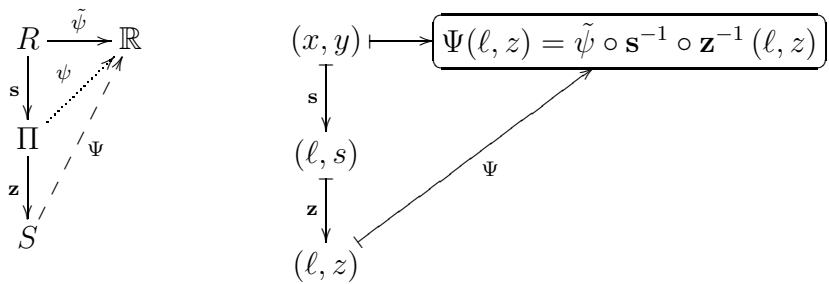
\[
\begin{array}{ccc}
\xymatrix@R=2em{ 
\mathlarger{R}
\ar[r]^{\tilde{\psi}} 
\ar[d]_{\mathbf{s}} 
& \mathlarger{ \mathbb{R}} \\ 
\mathlarger{\Pi}
\ar@{.>}[ur]^{\psi} 
\ar[d]_{\mathbf{z}}  & 
&   \\ 
\mathlarger{S} \ar@{-->}[uur]_{\Psi}  & 
}  
&

& 
\xymatrix{ 
(x,y)
\ar@{|->}[r]
\ar@{|->}[d]_{\mathbf{s}} 
& *+[F-:<3pt>]{ \Psi (\ell,z)=\tilde{\psi}  \circ \mathbf{s}^{ -1} \circ \mathbf{z}^{ -1}\left (\ell ,z\right ) }\\ 
(\ell,s)
\ar@{|->}[d]_{\mathbf{z}}  & 
&   \\ 
(\ell,z) \ar@{|->}[uur]_{\Psi}  & 
} 
\end{array}
\]

From this point forward, $W^{k,p}\left(D\right)$ denotes the Sobolev Space of elements in the Lebesgue Space
$L^p\left(D\right)$ on a domain $D\subset \mathbb{R}^2$ with generalized derivatives up to the order $k$, 
all of which belong to $L^p(D)$. We might recall that
\cite{mazya97, maz2013sobolev, bartle}:

\begin{definition}\label{defi:sobolevspaces}
A \emph{domain} is an open and connected subset $D\subset \mathbb{R}^2$ of the Euclidean space $\mathbb{R}^2$.
A distribution $g\in L^1\left(D\right)$ is a \emph{generalized derivative} of $f$ with respect to $x$ 
---also called \emph{weak} or \emph{distributional},
if for all analytic functions
 $\varphi$ with compact support in $D$, 
$\varphi \in C_0^{\infty}\left(D\right)$, we have: 
\[\int \!\!\!\! \int_{D} f \ \frac{\partial \varphi}{\partial x} dx \ dy=
-\int \!\!\!\! \int_{D} g \ \varphi \ dx \ dy.\]
Analogously, it can be defined for other coordinate systems and orders.
A necessary and sufficient condition for the density
of $C^{\infty}\left(\bar{D}\right)$
in a Sobolev Space $W^{k,2}\left(D\right)$ is unknown \cite[p. 10]{maz2013sobolev}.
However, it is enough for the domain $D$ to be a rectangle.
Therefore, the following results can be stated for a $\hat{f}\in C^{\infty}\left(\bar{D}\right)$
approximation of each distribution $f\in W^{k,p}\left(D\right)$.
\end{definition}

\begin{remark}\label{leibnitzrule}
As a particular case, Leibnitz Rule for product differentiation is valid in a non-empty open domain $D\subset \mathbb{R}^2$ 
when both factors and all the generalized derivatives involved are elements of $L^2(R)$ \cite[p. 11]{mazya97}.
Moreover, there is a generalized Green's Theorem \cite[p. 121]{necas67} 
that is valid for elements of the Sobolev Spaces
$W^{1,2}\left(D\right)$ in a bounded \emph{Lipschitz domain} \cite{maz2013sobolev} $D\subset \mathbb{R}^2$.
This allows the existence of a \emph{stream function}, defined in Theorem~\ref{theo:11}.
\end{remark}

\begin{definition}\label{defi1.2}
Let $L>\!\!>\!\!>h>0$, $R=[0,L]\times[0,h]$ and $\hat{R}=R\times[0,h]$.
If $\hat{\rho} \in L^1 \left( \hat{R} \times [0,\infty) ; (0,\infty)\right)$
such that
$\partial \hat{\rho}/\partial t=0$, 
 $\rho=\hat{\rho}|_{R}$, 
$\rho \in L^2\left( R ; \left(0,\infty\right)\right)$,
 $u\in L^2\left(R\right)$ with generalized derivatives 
 $\partial u/\partial x, \ \partial u/\partial y, \ \partial^2 u/\partial y^2 \in L^2\left(R\right)$;
 $v \in L^2\left(R\right)$ and
 $T \in  L^2\left(R; \left(0,\infty \right)\right)$ such that
$\partial T/\partial y, \ \partial^2 T /\partial y^2 \in L^2\left(R\right)$; 
 $\mu \in L^2\left(R\right)$,
 $p \in L^2\left(R\right)$, 
and $\kappa \in L^2\left(R\right)$,
to be the \emph{stationary density}, the horizontal and vertical \emph{velocity components},
the \emph{absolute temperature}, the \emph{dynamic viscosity}, the \emph{pressure}, 
and the \emph{thermal conductivity}, respectively.
Moreover, assume that
both products $\rho \ u, \ \rho \ v \in L^2\left(R\right)$,
and that all of them have first order generalized derivatives in $L^2\left(R\right)$.
This is, $\rho$, $u$, $v$, $T$, $\mu$,
$p$ and $\kappa$ are elements of the space $W^{1,2}\left(R\right)$.
\end{definition}

In $1942$, Dorodnityzn put forward a stationary gaseous boundary layer problem  \cite{Dorod42}  ---Eq. (\ref{eq:1}),
(\ref{eq:2}), (\ref{eq:3}), (\ref{eq:4}), (\ref{eq:5}),
(\ref{eq:7}), (\ref{eq:8}), and boundary conditions ---Eq.
(\ref{eq:12}), (\ref{eq:13}), (\ref{eq:14}), (\ref{eq:15}),
(\ref{eq:16}), (\ref{eq:17}),
in a long rec\-tangle $R=(0,L)\times (0,h) \in \mathbb{R}^2$ that represents the boundary layer region
for $L>\!>\!>h>0$. Dorodnitzyn's model is based on three simplified stationary Conservation of Mass,
Conservation of Momentum, and Conservation of Energy laws, Eq. (\ref{eq:1}),
(\ref{eq:2}) and (\ref{eq:3}),
  \begin{align}
  \frac{\partial \hspace{2pt}\left(\rho \hspace{2pt} u\right)}{\partial x}
  +
  \frac{\partial \hspace{2pt}\left(\rho \hspace{2pt} v\right)}{\partial y}
  & \hspace{2pt} = \hspace{2pt}
  0 \hspace{2pt}; \label{eq:1}\\
  \rho \left( u \hspace{2pt} \frac{\partial u}{\partial x}+
  v \hspace{2pt} \frac{\partial u}{\partial y}
  \right)
  & \hspace{2pt} =  \hspace{2pt}
  - \hspace{2pt} \frac{\partial p}{\partial x}
  +
  \frac{\partial}{\partial y}
  \left( \mu \hspace{2pt}
  \frac{\partial u}{\partial y}
  \right); 
  \hspace{7pt}\mbox{y} \label{eq:2}\\
  \rho \hspace{2pt}  
\left[ u \hspace{2pt}
\frac{\partial \hspace{2pt}\left(c_p \hspace{2pt} T\right)}{\partial x} 
+v \hspace{2pt}
\frac{\partial \hspace{2pt}\left(c_p\hspace{2pt} T\right)}{\partial y}\right]
& \hspace{2pt} = \hspace{2pt}
\frac{\partial}{\partial y} \left[\kappa \hspace{2pt} \frac{\partial T}{\partial y}\right]
+
\mu \hspace{2pt}
\left(
\frac{\partial u}{\partial y}
\right)^2
+
\frac{\partial p}{\partial t},
\label{eq:3}
\end{align}
for a \emph{stationary density} $\rho$, a horizontal and vertical \emph{velocity components}, $u$ and $v$,
an \emph{absolute temperature} $T$, a \emph{dynamic viscosity} $\mu$, a \emph{pressure} $p$, 
and a \emph{thermal conductivity} $\kappa$ whose main assumptions as elements of the Lebesgue space $L^2(R)$ 
are described in the Definition~\ref{defi1.2}. Under these assumptions, the complete system is made up of seven
identities in the Lebesgue space $L^1(R)$.

The value $c_p$ is the \emph{specific heat at constant pressure}. 
It is worth to notice that
there is a considerable difference between values of a \emph{gas constant for dry air} $\hat{R}_d=287$ \hspace{3pt}$[J\hspace{2pt}K^{-1}kg^{-1}]$,
and a \emph{gas constant for saturated water vapor} \cite[p. 1047]{bolton1980} \  $\hat{R}_v=461.50$ \hspace{3pt}$[J\hspace{2pt}K^{-1}kg^{-1}]$;
the \emph{specific heat at constant pressure for dry air}
\cite{Saha}
$c_{pd} = 1004 \hspace{2pt} [JK^{-1}kg^{-1}]$
and the \emph{specific heat at constant pressure for water vapor}
\cite{bolton1980} \ 
$c_{pv} = 1875 \hspace{2pt} [JK^{-1}kg^{-1}]$.
Therefore, one question that arises is if each model's solution will continuously vary with modifications of these constants
and what consequences does it have on the boundary layer separation.

Furthermore, we have
four Ideal Gas Thermodynamic Laws, Eq.  (\ref{eq:4}), (\ref{eq:5}),
(\ref{eq:7}), (\ref{eq:8}): the \emph{Prandtl number} $Pr=1$,
\begin{equation}\label{eq:4}
    Pr 
    \hspace{2pt} = \hspace{2pt}
    \frac{c_p \hspace{2pt} \mu}{\kappa}=1;
\end{equation}
the \emph{Equation of State} for the \emph{Universal Gas constant} $R^*$,
the volume $V$ of a rectangular prism 
$[0,L]\times [0,h]\times [0,h]\subset \mathbb{R}^3$ and
the \emph{number of moles} $n$ of an \emph{ideal gas} corresponding to the volume $V$,
\noindent
\begin{equation}\label{eq:5}
    p\hspace{2pt}V 
    \hspace{2pt}=\hspace{2pt}
    n\hspace{2pt }R^* \hspace{2pt} T;
\end{equation}
the \emph{adiabatic polytropic atmosphere} \cite[p. 35]{Tiet} where $b= 1.405$ and $c$  are constants,
\begin{equation}\label{eq:7}
       p\hspace{2pt} V^{b}  
       \hspace{2pt}= \hspace{2pt}
       c;
\end{equation}
and the \emph{Power-Law} \cite[p. 46]{SmitsDussauge2006}
\begin{equation} \label{eq:8}
 \frac{\mu}{\mu_h}  
 \hspace{2pt} = \hspace{2pt} 
  \left( \frac{T}{T_h}\right)^{\frac{19}{25}},
  \end{equation}
with boundary conditions, Eq.
(\ref{eq:12}), (\ref{eq:13}), (\ref{eq:14}), (\ref{eq:15}),
(\ref{eq:16}), (\ref{eq:17}):
\begin{align}
(u,v)|_{\{(x,h)\colon 0\leq x \leq L\}}
& \hspace{2pt} = \hspace{2pt}
(-U,0),  \label{eq:12}\\
(u,v)|_{\{(x,0)\colon 0 \leq x \leq L\}}
& \hspace{2pt} = \hspace{2pt} 
(0,0), \label{eq:13}
\end{align}
for a positive value of the \emph{free-stream velocity}, $U>0$,
the \emph{no slip condition} at the surface,
a \emph{free-stream temperature} $T_h>0$, a \emph{free-stream dynamic viscosity} $\mu_h>0$,
\begin{align}
T|_{\{(x,h)\colon 0\leq x \leq L\}}
& \hspace{2pt} = \hspace{2pt}
T_h>0,  \label{eq:14}\\
\mu|_{\{(x,h)\colon 0\leq x \leq L\}}
& \hspace{2pt} = \hspace{2pt}
\mu_h>0, \label{eq:15}
\end{align}
and a Neumann condition:
 \begin{equation}\label{eq:17}
\frac{\partial T}{\partial y}\bigg|_{\left\{(x,0)\colon 0 \leq x \leq L\right\}}
\hspace{2pt} = \hspace{2pt}
 0.
 \end{equation}
 In \cite{val08}, periodic conditions, such as the ones used in Chupin and Sart's work \cite{ChupinSart},
 were included at the vertical sections of the \emph{topological boundary} $\partial R$, such that for all $y\in[0,h]$:
\begin{equation}\label{eq:16}
    \left(u\left(0,y\right),0\right)
\hspace{2pt} = \hspace{2pt} 
  \left(u\left(L,y\right),0\right).
\end{equation}
\noindent
It is worth to remark the fact that both these laws and their boundary conditions are satisfied in
atmospheric conditions. For example, the no slip condition, Eq. (\ref{eq:13}) is verified for values of $u$ under 
the speed of sound.

Dorodnitzyn took the first gaseous boundary layer model, stated by Busemann in $1935$ \cite{busemann1935}
and his idea to express absolute temperature $T$ in terms of the horizontal component $u$ of velocity, 
but included a term $\partial p/\partial x$, 
not present in Busemann's model, which could lead to a boundary layer separation.
Busemann used a different power-law exponent in Eq. (\ref{eq:8}), which was later corrected in
von K\'{a}rm\'{a}n and Tsien's article of $1938$ \cite{VKar1938}.
Instead, he shows how to eliminate $\partial p /\partial x$ when the \emph{free stream velocity}, 
$U$ of Eq. (\ref{eq:12}), is constant.
In order to do this, he expressed the Conservation of Energy Law, Eq. (\ref{eq:3}), 
in terms of the \emph{total energy per unit mass}, $E =c_pT+u^2/2$,
as Luigi Crocco did in $1932$ \cite{Crocco32}.  

As a result, Eq. (\ref{eq:3}) is substituted by Eq. (\ref{eq:18}),
and the system of equations becomes
Eq. (\ref{eq:1}),
(\ref{eq:2}), (\ref{eq:18}), (\ref{eq:4}), (\ref{eq:5}),
(\ref{eq:7}), (\ref{eq:8}).
Another consequence is that the constant $E=c_pT_h+U/2$, 
given in terms of (\ref{eq:12}) and (\ref{eq:14}),
is a solution of Eq. (\ref{eq:18}) that satisfies its boundary conditions.
This makes possible to express the absolute temperature $T$ in terms of the horizontal velocity $u$,
and, to reduce the model to a system of two conditions for the \emph{stream function}
$\tilde{\psi}\in C^1(R)$, as it is proved in Theorem~\ref{theo:11}.

Luigi Crocco's Procedure, described in the original article \cite{Crocco32},
can be applied to the distributions $\rho, u,v,T,p,\kappa,\mu$ because 
the generalized derivatives of the variables
are elements of the Lebesgue space $L^2\left(R\right)$, and we can proceed as we would
with classical derivatives  to apply a generalized Leibnitz Rule for the product
---as stated in Remark~\ref{leibnitzrule} and \cite[p. 11]{mazya97}, so that Eq. (\ref{eq:3}) is satisfied
if and only if:
\begin{equation}\label{eq:18}
\rho \hspace{2pt}  
\left[ u \hspace{2pt}
\frac{\partial }{\partial x} 
+v \hspace{2pt}
\frac{\partial}{\partial y}\right]
\left(
c_p \hspace{2pt} T+\frac{u^2}{2}
\right)
\hspace{2pt} = \hspace{2pt}
\frac{\partial}{\partial y} \left[ \mu \hspace{2pt} 
\frac{\partial \hspace{2pt} }{\partial y}
\left(
c_p \hspace{2pt} T
+
\frac{u^2}{2}
\right)
\right].
\end{equation}
Moreover,
    $T(u)=T_0\hspace{2pt}
    \left(1-u^2/(2c_p\hspace{2pt}T_0)\right)$
where $T_0  =T_h+1-(U^2/2c_p)>0$ and $i_0  =c_p T_0>0$.
If we take into account the \emph{atmospheric pressure} expression
$p(x,\hat{y})=g\int_{\hat{y}}^{\infty}\rho (x,y)\ dy$ for the \emph{standard gravity}
$g$ and a linear decrease 
$T(x,y)=T_0-\beta y$ for a constant $\beta>0$ 
for $(x,y)\in R$, then:
\[p
 \hspace{2pt} \cong \hspace{2pt}
c_1
\hspace{2pt}
    \left[1-\left(U^2/2i_0\right)\right]^{\frac{b}{(b-1)}}\]
and    the density
$\rho(u)
 \hspace{2pt}  \cong \hspace{2pt}
      c_2
     \hspace{2pt}
      \left[1-\left(U^2/2i_0\right)\right]^{\frac{b}{(b-1)}}/
      \left[ 1-\left( 
      u^2\left(x,y\right)/2i_0\right)\right]$.
      From Eq. (\ref{eq:5}),
the dynamic viscosity
$\mu(u)
 \hspace{2pt} = \hspace{2pt}
     c_3
     \hspace{2pt}
     \left[1-\left(u^2/2i_0\right)\right]^{\frac{19}{25}}$
for a gas constant $\hat{R}=R^*/M$, the molecular weight $M$,
$p_0  =\left(n\hspace{2pt}R^*\hspace{2pt}T_0\right)/V>0$,
$c_1 = 
 p_0 \hspace{2pt} T_0^{\frac{2b}{b-1}}$,  
 $c_2 = 
 c_1 \hspace{2pt}\hat{R}^{-1} \hspace{2pt}T_0^{-1}$,
$c_3 = 
\mu_h \hspace{2pt} T_h^{-\frac{19}{25}}\hspace{2pt} T_0^{\frac{19}{25}}$.

\begin{theorem}\label{theo:11}
Let $\rho$, $u$, $v$, $T$, $p$, $\kappa$, $\mu$ be as in Definition~\ref{defi1.2}.
Assume $p= c_1
\hspace{2pt}
    \left[1-\left(U^2/2i_0\right)\right]^{\frac{b}{(b-1)}}$,
    $\partial u/\partial x=0$,
    and that the variables verify Eq.  (\ref{eq:1}),
(\ref{eq:2}), (\ref{eq:18}), (\ref{eq:4}), (\ref{eq:5}),
(\ref{eq:7}), (\ref{eq:8}) and
(\ref{eq:12}), (\ref{eq:13}), (\ref{eq:14}), (\ref{eq:15}),
(\ref{eq:16}), (\ref{eq:17}). 
Consider 
$\xymatrix@1{
R \ar[r]^-{\mathbf{s}} & \Pi}$,
$\xymatrix@1{
(x,y) \ar@{|-{>}}[r]^-{\mathbf{s}} & (\ell, s)}$, where:
 \[\xymatrix@1{
\ell\left(\hat{x},\hat{y}\right) \ar@{|-{>}}[r] &  \int_{0}^{\hat{x}} p\left(x,\hat{y}\right) 
  \hspace{2pt}dx}\] and
 \[\xymatrix@1{
s\left(\hat{x},\hat{y}\right) \ar@{|-{>}}[r] &  \int_0^{\hat{y}} \rho\left(\hat{x},y \right)
  \hspace{2pt}dy}.\] Denote $\Pi=\mathbf{s}\left(R\right)$ and $\sigma_0=1-U^2/(2i_0)$.
Then, there is a \emph{stream-function} $\tilde{\psi} \in W^{2,2}\left(R\right)$ 
such that
$\partial \tilde{\psi}/\partial x=-\hspace{2pt}\rho\hspace{2pt} v$,
$\partial \tilde{\psi}/\partial y=\hspace{2pt}\rho\hspace{2pt} u$, 
and a $\tilde{\sigma } =1-\left(u^2/2i_0\right)\in W^{1,2}\left(R;\left(0,\infty\right)\right)$, 
such that
$\psi \colon=\tilde{\psi} \ \circ \ \mathbf{s}^{-1}\in W^{2,2}\left(\Pi\right)$ 
and $\sigma\colon = \tilde{\sigma} \circ \mathbf{s}^{-1}\in W^{1,2}\left(\Pi;\left(0,\infty\right)\right)$ satisfy:
\begin{equation}\label{eq:2.55}
\frac{\partial \psi}{\partial s}
    \frac{\partial^2 \psi}{\partial l \partial s}
    -
    \frac{\partial \psi}{\partial l}
    \frac{\partial^2 \psi}{\partial s^2}
    =
c_1^{-1} \hspace{2pt} c_2 \hspace{2pt}    
   c_3\hspace{2pt}
\hspace{2pt}
\sigma_0^{\frac{b}{(b-1)}-1}
\hspace{2pt}
\frac{\partial}{\partial s}
\left[\sigma^{\frac{19}{25}-1}
\hspace{2pt}
\frac{\partial^2 \psi}{\partial s^2}\right].
\end{equation}  
\end{theorem}

\begin{proof}
First, we describe Dorodnitzyn's diffeomorphism:
The new domain's, $\Pi$, extremes are $\ell_M=\ell(0,L)=c_1\ \sigma_0\ L$
and $s(x,h)=c_2\ \sigma_0^{b/(b-1)-1}\ h$.
The partial derivatives of $\ell$ over $R$ are $\partial \ell/\partial x=c_1\sigma_0$
and $\partial \ell/\partial y=0$.
Given that $\partial u/\partial x=0$, the \emph{Dominated Convergence Theorem} \cite[p. 44]{bartle}
implies that $\partial s/\partial x=0$.
Moreover, $\partial s/\partial y=\rho$.
This may distinguish that $s$ is the \emph{entropy} \cite[p. 432]{evans2004}
and that Dorodnitzyn's statment of the problem is, in fact, an \emph{entropy method}.
The  \emph{Jacobian determinant} $|D\mathbf{s}|=c_1 \hspace{2pt} \sigma_0 \hspace{2pt} \rho >0$.
Thus, the \emph{Inverse Function Theorem} 
implies that $\mathbf{s}$ is a
\emph{diffeomorphism} that
takes the rectangle $R$ into a polygonal domain $\Pi$.
In this coordinate system, \emph{von K\'{a}rm\'{a}n's Integral Formula}
for a compressible fluid in $R$ has an incompressible form in $\Pi$ \cite[p. 258]{krasnov852}.

Because of the zero divergence given in Eq. (\ref{eq:1}),
the generalized Green Theorem for Sobolev Spaces
$W^{1,2}\left(R\right)$ 
on a rectangular domain $R$ \cite[p. 121]{necas67}
and the Poincar\'{e} Lemma allow us to define a
\emph{stream function} $\tilde{\psi}\in W^{2,2}\left(R\right)$,
$\tilde{\psi}_{(0,0)} (x,y)=
\int_{(0,0)}^{(x,y)}-\rho \hspace{2pt} v \hspace{2pt}dx$.
The \emph{stream function} $\tilde{\psi}$ 
is regarded in $\Pi$ as $\psi\in W^{2,2}\left(\Pi \right)$.
Once more, over the rectangular domain $R$, we can apply the
Leibniz Rule for $L^2$-distributions of Remark~\ref{leibnitzrule} to see that,
in terms of $\psi$, the system has an incompressible non-linear 
expression for the convective derivative term in the left hand side of Eq. (\ref{eq:2}) in $\Pi$ as:
$$  \rho 
  \left( 
   u \frac{\partial u}{\partial x}
   \hspace{3pt} + \hspace{3pt}
  v \frac{\partial u}{\partial y}
  \right)
    =  c_1 \hspace{2pt} \sigma_0 \hspace{2pt} \rho
        \left(\frac{\partial \psi}{\partial s}
    \frac{\partial^2 \psi}{\partial \ell \partial s}
    -
    \frac{\partial \psi}{\partial \ell}
    \frac{\partial^2 \psi}{\partial s^2}\right).$$
This way, it is possible to cancel the density $\rho$ factor with with its correspondent 
right hand side of Eq. (\ref{eq:2}) written in $\Pi$ as:
$$ \frac{\partial}{\partial y}
  \left( \mu \hspace{2pt}
  \frac{\partial u}{\partial y}
  \right)
  =
  c_3
      \frac{\partial}{\partial y} 
  \left[
      \sigma^{\frac{19}{25}}
      \frac{\partial }{\partial y}
      \left(
      \frac{1}{\rho}
      \frac{\partial \tilde{\psi}}{\partial y}
      \right)
      \right]
      =
     c_2 \hspace{2pt} c_3 \hspace{2pt} \sigma_0^{\frac{b}{b-1}} \hspace{2pt} \rho 
     \hspace{2pt}
\frac{\partial}{\partial s}
\left[\sigma^{\frac{19}{25}-1}
\hspace{2pt}
\frac{\partial^2 \psi}{\partial s^2}\right],     
$$
where $\partial p/\partial x=0$ because $p$ is constant in $R$, and
$\sigma$
quantifies the amount 
of kinetic energy is transformed into heat \cite{krasnov852}.
As a distribution, $\tilde{\sigma}(u)\in W^{1,2}\left(R;\left(0,\infty\right)\right)$
and $\partial^2 \tilde{\sigma}/\partial y^2 \in L^2(R)$ directly from $T\in W^{1,2}\left(R\right)$
and the generalized deriviative of order two, $\partial^2 T /\partial y^2 \in L^2\left(R\right)$.    
Therefore, under the hypothesis of Definition~\ref{defi1.2} over the variables,
the original problem of Eq. (\ref{eq:1}),
(\ref{eq:2}), (\ref{eq:3}), (\ref{eq:4}), (\ref{eq:5}),
(\ref{eq:7}), (\ref{eq:8}), 
is transformed into the condition given by Eq. (\ref{eq:2.55})
with inherited boundary conditions.
\end{proof}

At this point, Dorodnitzyn adapts Blasius' normalization $\mathbf{z}$ to express Eq. (\ref{eq:2.55}) 
as an the Ordinary Differential Eq. (\ref{eq:22}),
which he transforms into the Quasi-Linear Parabolic Eq. (\ref{eq:detau}).

\begin{theorem}\label{theo:22}
Under the same hypotheses of Theorem \ref{theo:11},
let \[S =
\left\{(\ell,z)\in \mathbb{R}^2 \hspace{4pt}|\hspace{4pt}(\ell,s)\in \Pi
\hspace{5pt} and \hspace{5pt} z =s/\ell^{1/2}\in (0,\infty)\right\},\]
 $\xymatrix@1{
\Pi \ar[r]^-{\mathbf{z}} & S}$,
$\xymatrix@1{
(\ell,s) \ar@{|-{>}}[r]^-{\mathbf{z}} & \left(\ell, z\right)}$,
 $\xymatrix@1{
z\left(\ell,s\right) \ar@{|-{>}}[r] &  s/\sqrt{\ell}}$, 
$\Psi \colon =\tilde{\psi} \ \circ \mathbf{s}^{ -1} \circ \mathbf{z}^{ -1} \in W^{4,1}(S)$
such that $\Psi = f(z) \ g(\ell)$,
$u_s \colon =u \circ \mathbf{s}^{ -1} \circ \mathbf{z}^{ -1}$,
and \[\tau_s \colon =
\left (1 -u_s^{2}/(2i_{0})\right )^{ -6/25}
\partial ^{2}f/ \partial z^{2}.\]
Then, 
\begin{equation}\label{eq:detau}
\tau_s \ \frac{ \partial ^{2} \tau_s}{ \partial u_s^{2}} 
=
-
\frac{1}{2}\ c_1 \hspace{2pt} c_2^{-1} \hspace{2pt}    
   c_3^{-1}\hspace{2pt}
\hspace{2pt}
\sigma_0^{1-\frac{b}{(b-1)}}
\hspace{2pt} \ 
 u_s \
\left(1-\frac{u_s^2}{2i_0}\right)^{-6/25}.
\end{equation}
\end{theorem}

\begin{proof}
First of all, the \emph{Jacobian determinant} $|D\mathbf{z}|= \ell^{-1/2} >0$ for all $\ell>0$.
Therefore, $\mathbf{z}$ is a \emph{diffeomorphism} from $\Pi$ to $S$. 
Suppose $\Psi=g \cdot f$ is separable as the product of two distributions,
independently determined by the variables $\ell$ and $z$,
such that 
$\partial\Psi/\partial z
=
l^{1/2} \ \partial f/\partial z$.
Then, the Leibniz Rule for a product \cite[p. 149]{Bressan13}
of $g\in C^1(S)$ and an integrable distribution $f \in L^1_{loc}(S)$ 
over an open set $S\not= \varnothing$, applied to
$\partial(g \cdot f)/\partial z$
and the condition 
$\partial\Psi/\partial z
=
\ell^{1/2} \ \partial f/\partial z$
imply that  $g(\ell) =\ell^{1/2}$.

Second, if $\psi\in W^{2,2}\left(\Pi\right)$ is a weak solution of Eq.(\ref{eq:2.55}),
then $f\in  W^{4,1}\left(S \right)$, such that $\partial^k\Psi/\partial z^k
=
l^{1/2} \ \partial^k f/\partial z^k$ for $k\in \{1,2,3,4\}$,
is a weak solution to the Ordinary Differential Equation:
\begin{equation}\label{eq:22} 
-\frac{1}{2}\ f\ \frac{ \partial ^{2}f}{ \partial z^{2}} =
c_1^{-1}\ c_2 \ c_3\ \sigma_0^{\frac{b}{(b-1)}-1} \
\frac{ \partial }{ \partial z}\left ( 
\sigma_s^{-\frac{6}{25}}\frac{ \partial ^{2}f}{ \partial z^{2}}\right ) ,
\end{equation}
where $\sigma_s =\sigma  \circ \mathbf{s}^{ -1} \circ \mathbf{z}^{ -1}$.
In order to verify this, we write the left and right side of Eq.(\ref{eq:2.55}) in terms of the new coordinates. 
The left side becomes:
\begin{equation}\label{eq:23} 
\frac{ \partial \psi }{ \partial s}\frac{ \partial ^{2}\psi }{ \partial \ell \partial s} -\frac{^{} \partial \psi }{ \partial l}\frac{ \partial ^{2}\psi }{ \partial s^{2}} = 
-\frac{1}{2}\ \ell^{ -1}\ f  \ \frac{ \partial ^{2}f}{ \partial z^{2}};
\end{equation}
and, the right side is:
\begin{equation}\label{eq:24} 
\frac{ \partial }{ \partial s}\left(\sigma ^{ -6/25}\ \frac{ \partial ^{2}\psi }{ \partial s^{2}}\right)\ \frac{ \partial }{ \partial s^{2}} =\ \ell^{ -1}\ \frac{ \partial }{ \partial z}\left(\sigma ^{ -6/25}\ \frac{ \partial ^{2f}}{ \partial z^{2}}\right). 
\end{equation}
This way, the factor $\ell^{-1}$ is nullified 
when Eq. (\ref{eq:23}) is equal to Eq. (\ref{eq:24}) and we obtain Eq. (\ref{eq:22}).  

Third, let $u_s=u \circ \mathbf{s}^{-1} \circ \mathbf{z}^{-1}$, then $ f(z)=\int_{0}^{z} \ u_s \left(l,z'\right) \ dz'$:
From the \emph{stream-function}'s separation of the first step, we have
$\partial f/\partial z=\ell^{-1/2}\ \partial \Psi/\partial z$.
Moreover, if $f\in W^{1,1}(0,\infty)$, then
$f(z) = f(0)+\int_0^{z}\ \partial f/\partial z\ (z')\ dz'$.
Because of $\tilde{\psi}(0,0)=0$, $\psi(0,0)=\Psi(0.0)=f(0)=0$ and $f(z)= \int_0^{z}\ \partial f/\partial z\ (z')\ dz'$.
In addition, for each $(\ell,z)\in S$: $\partial \Psi/\partial z \ (\ell,z) =
\ell^{1/2}\ 
\partial \psi/\partial s \
\left(\mathbf{z}^{-1}\left(\ell,z\right)\right)\
=
\ell^{1/2}\
(1/\rho)\
\partial\tilde{\psi}/\partial y\ 
\left(\mathbf{s}^{-1}\left(\mathbf{z}^{-1}\left(\ell,z\right)\right)\right)\
=
\ell^{1/2}\ u\ \circ \ \mathbf{s}^{-1}\ \circ \ \mathbf{z}^{-1}\ \left(\ell,z\right)$. 
This is, 
$\partial\Psi/\partial z\left(\ell,z\right)
=
\ell^{1/2}\ 
u_s\ \left(\ell,z\right)$.

As a direct consequence of both relations,
$f(z)
     =  \int_{0}^{z}\ u_s(\ell,z')\ dz'$ and 
      $\partial u_s/\partial z
=
\partial^2 f/\partial z^2$
Finally, if $\tau_s  =
\left (1 -u_s^{2}/(2i_{0})\right )^{ -6/25}
\partial ^{2}f/ \partial z^{2}$, then: 
\[\partial ^{2}f/ \partial z^{2}=\left (1 -u_s^{2}/(2i_{0})\right )^{6/25}\tau_s  ;\]
 the left side of Eq. (\ref{eq:22}) is:
\begin{eqnarray*}
-\frac{1}{2}\ f\ \frac{\partial^{2}f}{\partial z^{2}}
&=&
-\frac{1}{2}\ \left(\int_{0}^{z} \ u_{s}\left(l,z'\right)\ dz'\right)\ 
\left(1-\frac{u_s^{2}}{2i_{0}}\right)^{6/25}\ \tau_s;
\end{eqnarray*}
and the right side of Eq. (\ref{eq:22}) becomes:
\[\frac{\partial \tau_s}{\partial z}
= \frac{\partial \tau_s}{\partial u_s}
\frac{\partial u_s}{\partial z}
=\left(1-\frac{u_{s}^{2}}{2i_{0}}\right)^{6/25}\
\tau_s\ 
\frac{\partial \tau_s}{\partial u_s}.\]
Thus, Eq. (\ref{eq:22}), in terms of $\tau_s$ and $u_s$, allows the elimination of the factor $(\sigma^{6/25} \ \tau_s)$,
present on both sides:
\begin{equation}\label{eq:prev}
-\frac{1}{2}\ \int_{0}^{z} \ u_{s}\left(l,z'\right)\ dz' \ (\sigma^{6/25} \ \tau_s) \ 
=
c_1^{-1} \hspace{2pt} c_2 \hspace{2pt}    
   c_3\hspace{2pt}
\hspace{2pt}
\sigma_0^{\frac{b}{(b-1)}-1}
\hspace{2pt} \ (\sigma^{6/25}  \tau_s) \ 
\frac{\partial \tau_s}{\partial u_s}.
\end{equation}
A derivation with respect to  $z$ on both sides of Eq. (\ref{eq:prev}) gives Ec. (\ref{eq:detau}).
\end{proof}

\section{Reynolds' Shear Stress Limit for Dorodnitzyn's Boundary Layer}

\begin{theorem}\label{theo:33}
Under the same hypotheses of Theorem \ref{theo:22},
let
$\xymatrix@1{
R \ar[r]^-{\phi^{\epsilon}} & R^{\epsilon}}$
for $\epsilon =h/L>0$, where
$\xymatrix@1{
(x,y) \ar@{|-{>}}[r]^-{\phi^{\epsilon}} & \left(x/L,y/\left(L\epsilon\right)\right)}$,
$\left(x/L,y/\left(L\epsilon\right)\right)= \left(x^{*},y^{*}\right)$.
Furthermore, assume that $\partial u/\partial y(x,y)>0$ for each \emph{Lebesgue point} $(x,y)\in R$.
Then, there is a limit $u^*\in W^{1,2}(R)$, $u^*=\lim_{\epsilon \to 0}\ u^{\epsilon}$ of 
$u^{\epsilon}=\left(1/L\right)u$, such that:
\begin{equation}\label{eq:dereynolds}
 \frac{\partial }{\partial  u_s}
\left(1-\frac{(u^*)^2}{2i_0}\right)^{19/25}=0.
\end{equation}
\end{theorem}

\begin{proof}
Let $\sigma^{\epsilon}=
      1-\left( 
      \left[Lu^{\epsilon}\right]^2/2i_0\right)$,
     $\tau_s=\left(1-u_s^2/\left(2i_0\right)\right)^{(19/25)-1}\ \partial u_s/\partial z
      =\tilde{c}\ x^{1/2}\ \tau$, 
where $\partial u_s/\partial z=\ell^{1/2}\rho^{-1}\partial u/\partial y$,
$\tilde{c}=c_1^{1/2}\ c_2^{-1}\ \sigma_0^{1/2-b/(b-1)}$, and $\tau=\mu \ \partial u/\partial y$.
Thus, Eq. (\ref{eq:detau}), in terms of $\epsilon$, becomes:
\[
\epsilon \  \tilde{c}\ x^{1/2}\left(\sigma^{\epsilon}\right)^{19/25}\frac{\partial \tau_s}{\partial y}
\left(\frac{\partial u^{\epsilon}}{\partial y^{*}}\right)^{-1}
-\epsilon^2 \ \left(\frac{\partial \tau_s}{\partial y}\right)^2
\left(\frac{\partial u^{\epsilon}}{\partial y^{*}}\right)^{-2}
=\]
\begin{equation}\label{eq:epsilon}
=
-
\frac{1}{2}\ c_1 \hspace{2pt} c_2^{-1} \hspace{2pt}    
   c_3^{-1}\hspace{2pt}
\hspace{2pt}
\sigma_0^{1-\frac{b}{(b-1)}}
\hspace{2pt} \ 
 u_s \
\left(1-\frac{u_s^2}{2i_0}\right)^{-6/25}.
\end{equation}
In a previous article \cite{val08}, we showed that, under these circumstances, 
$ \Vert \nabla u^{\epsilon} 
 \Vert_{L^2\left(R\right)}
\leq
(c_2\hspace{2pt}U^3)/(2\hspace{2pt}C)$
for a constant $C$ that is independent of the parameter $\epsilon$.
This way, the sequence $\left(u^{\epsilon}\right)$
is bounded in the Sobolev Space $W^{1,2}\left(R\right)$.
Then, the Rellich-Kondrachov compactness theorem
\cite[p. 173, 178]{Bressan13} implies that there is
a subsequence  that
converges strongly in
$L^2\left(R\right)$,
and the sequence $\partial u^{\epsilon}/\partial y^*$
converges weakly in $L^2\left(R\right)$
to a gene\-ralized derivative
$\partial u^*/\partial y^*$
of the limit 
$u^* \in L^2\left(R\right)$.
Hence, $u^*$ 
is a weak solution of  Eq. (\ref{eq:epsilon}), in 
$L^2\left(R\right)$
when the parameter $\epsilon$ tends to $0$.
\end{proof}

\begin{corollary}\label{cor:33}
The limit 
$u^* =U$.
is a constant solution of  Eq. (\ref{eq:epsilon}).
Moreover, if $\partial u/\partial y\cong U/h$, then
\[\tau_s \cong c_1^{1/2}\ c_2^{-1}\ \frac{U}{h}\ x^{1/2} \
\left(1-\frac{U^2}{2c_{p}\ T_0}\right)^{1-\frac{b}{b-1}-\frac{6}{25}+\frac{1}{2}},\]
and we obtain the shear stress estimate:
\[\tau^{*} = \frac{U}{h}\ \left(1-\frac{U^2}{2c_{p}\ T_0}\right)^{19/25}.\]
\end{corollary}

\section*{Conclusion}
It is possible to deduce approximate shear stress formulas 
from the Dorodnitzyn's gaseous boundary layer model
and a Reynolds' Limit Formula developed through a small parameter
statement of the problem without taking away the convective derivative 
non-linear term of the conservation of momentum equation.
These estimates provide a new family of deterministic
atmospheric boundary layer separation
indicators to be tested.
The simplicity of its calculations and interpretations in terms of
its boundary conditions and variation of specific heat at constant pressure coefficients
may allow a wide range of analyses with local data
and free from computational time requirements. 

\section*{Acknowledgements}
I would like to express my deepest and sincere gratitude to Dr. Valeri Kucherenko, who
trained me, taught me, showed me the light at the end of the tunnel in many occasions, 
supported and encouraged me to become a scientist and a mathematician.

\section*{Conflict of interest}
The author declares no conflict of interest.






\bibliographystyle{elsarticle-num-names}
\bibliography{elsarticle-template-1-num.bib}







\end{document}